\documentclass{amsart}
\usepackage{graphicx} 
\usepackage{amsmath, amsfonts, amsthm, amssymb}
\usepackage{geometry}
\usepackage{enumitem}
\usepackage[english]{babel}
\usepackage{indentfirst}
\usepackage{hyperref}
\usepackage{xcolor}
\theoremstyle{definition}
\newtheorem{definition}{Definition}[section]
\newtheorem{theorem}[definition]{Theorem}
\newtheorem{question}[definition]{Question}
\newtheorem{lemma}[definition]{Lemma}

\DeclareMathOperator{\andd}{\,and\,}
\DeclareMathOperator{\Fn}{Fn}
\DeclareMathOperator{\dom}{dom}
\DeclareMathOperator{\PP}{{\mathbb P}}
\DeclareMathOperator{\PAT}{\mathbf{PAT}}
\DeclareMathSymbol{\mlq}{\mathord}{operators}{``}
\DeclareMathSymbol{\mrq}{\mathord}{operators}{`'}

\date{October 2024}

\begin{document}
\title{On Products of $\Delta$-sets}
\author{Rodrigo Rey Carvalho}
\address{Independent Researcher}
\email{rodrigo.rey.carvalho@gmail.com}

\author{Vinicius de Oliveira Rodrigues}
\address{University of São Paulo}
\email[Corresponding author]{vinior@ime.usp.br}
\begin{abstract}We prove the consistency of the existence of a $Q$-set whose square is not a $\Delta$-set and that if there is a $\Delta$-set, then there exists a $\Delta$-set whose all finite powers are $\Delta$-sets.\vspace{1em}

    
\end{abstract}

\maketitle

\section{Introduction}\label{sec:intro}

$\Delta$-sets and $Q$-sets are well-studied classes of special uncountable subsets of reals which have strong and interesting properties of category.

The history of $Q$-sets goes back to Hausdorff \cite{HausdorffQ}, Sierpinski \cite{SierpinskiQ} and Rothberger \cite{RothbergerQ}. A $Q$-set is a subspace of the real line - or some other Polish space - whose every subspace is a relative $G_\delta$ (see e.g. \cite{miller1984special}). $Q$-sets are related to the Normal Moore Space Problem for the separable case: there exists an uncountable non-metrizable Moore normal space if, and only if there exists a $Q$-set (see e.g. \cite{hernandez2005q}). The existence of $Q$-sets is independent of ZFC - one readily sees that the existence of a $Q$-set implies $2^{\omega_1}=\mathfrak c$, and every uncountable subset of the reals of size less than $\mathfrak p$ is a $Q$-set \cite{brendle1999dow}.

The study of $\Delta$-sets started with Reed and van Douwen, as outlined in \cite[p. 176]{reed1990set}, and is related to the study of the dividing line between the class of normal spaces and the class of countably paracompact spaces, thus, of Dowker spaces. A vanishing sequence is a countable decreasing sequence of sets with empty intersections. We say that an uncountable subset $X$ of the reals- or some other Polish space - 
is a $\Delta$-set if for every vanishing sequence $(F_n: n \in \omega)$ of subsets of $X$, there exists a vanishing sequence of open sets (equivalently, $G_\delta$ sets) of $X$ $(U_n: n \in \omega)$ such that for every $n \in \omega$, $F_n\subseteq U_n$. We call this an open expansion of the vanishing sequence. Przymusinski has proved that there exists an uncountable non-metrizable Moore normal space if, and only if there exists a $\Delta$-set \cite{przymusinski1977normality}.


Clearly, every $Q$-set is a $\Delta$-set. In \cite{fakeKnight}, it is claimed that it is consistent that there exists a $\Delta$-set that is not a $Q$-set. However, the proof is considered to be somewhat unclear and alternative proofs are looked for \cite{leiderman2023delta}. It is not known if the existence of $\Delta$-sets implies that $2^{\omega_1}=2^\omega$.

In this paper, we prove that two classical results regarding the productivity of $Q$-sets also hold for $\Delta$-sets. A natural question is whether the finite product of $Q$-sets or $\Delta$-sets is a $Q$-set or $\Delta$-set.

In \cite{fleissner1983squares}, Fleissner claimed to have proven the consistency of the existence of is a $Q$-set whose square is not a $Q$-set. However, Miller observed that Fleissner’s proof is flawed \cite{miller2006hodgepodge}, and Brendle subsequently fixed the proof \cite{BrendleQ}. In Section \ref{sec:forcing}, we modify Brendle's argument to prove the consistency of the existence of a $Q$-set whose square is not a $\Delta$-set.

In \cite{przymusinski1980existence}, Przymusinski proved that the existence of a $Q$-set implies the existence of a $Q$-set whose all finite powers are also $Q$-sets. Such a $Q$-set is called a ``strong $Q$-set''. In Section \ref{sec:strongdelta} we prove that the existence of a $\Delta$-set implies the existence of a \textit{strong} $\Delta$-set, that is, of a $\Delta$-set whose all finite powers are $\Delta$-sets.

\section{A \texorpdfstring{$Q$}{Q}-set whose square is not a \texorpdfstring{$\Delta$}{Delta}-set.}\label{sec:forcing}

In \cite{fleissner1980}, Fleissner and  Miller defined a forcing notion designed to turn a subspace $A$ of some $X\subseteq 2^\omega$ into a $F_\sigma$. It may be defined as follows:

\begin{definition}
Assume $A\subseteq X\subseteq 2^\omega$. We define $P(A, X)$ to be the set of all $q\in [2^{<\omega}\times\{0\}\times \omega \cup A\times\{1\}\times \omega]^{<\omega}$ such that:

\begin{equation*}
    \forall n \in \omega \,\forall a \in A\, \forall s \in 2^{<\omega} \, (s, 0, n) \in r \andd (a, 1, n) \in r \implies s\not \subseteq a \label{eq:condition} \tag{$*$}
\end{equation*}

$P(A, X)$ is ordered by reverse inclusion.
\end{definition}

Intuitively, the poset is trying to define open sets $U_n$ so that $X\setminus A=X\cap\bigcap_{n \in \omega} U_n$.  $(s, 0, n)\in q$ is interpreted as ``$q$ promises that $[s] \subseteq U_n$'', and $(a, 1, n) \in q$ is interpreted as ``$q$ promises that $a\notin U_n$, thus, (\ref{eq:condition}) is required for $r$ to be internally consistent. Informally, if $G$ is $\mathbb P(A, X)$-generic, one defines $U_n=\bigcup\{[s]: (s, 0, n) \in q \in G\}$, and standard density arguments are used to show that $X\setminus A=X\cap\bigcap_{n \in \omega} U_n$.

Thus, by starting with any set $X$ and iterate this forcing notion collecting all the names for subsets of $X$ along the iteration, one may transform any set $X$ in the ground model into a $Q$-set. In \cite{fleissner1980}, Fleissner and Miller used a modification of this technique to construct a $Q$-set that is concentrated in the rationals: he added $\omega_1$ Cohen reals and iterated these forcing notions including the rationals in all sets $A$.

In this section, we modify the construction in \cite{BrendleQ} to prove that it is consistent that there exists a $Q$-set whose square is not a $\Delta$-set. More specifically:

\begin{theorem}Assuming GCH, for every uncountable cardinal  $\kappa$ (in the ground model), there exists a cardinal-preserving forcing notion $\mathbb P$ such that: $$\Vdash_\mathbb P ``\check \kappa\leq \mathfrak c \andd (\exists X\subseteq 2^\omega\, |X|=\check \kappa \andd X \text{ is a }Q\text{-set} \andd X^2 \text{ is not a } \Delta\text{-set})".$$
\end{theorem}

In some sense, the model we use is the same as the one used in \cite{BrendleQ}: one starts with GCH, then adds $\kappa$ Cohen reals, and then transforms a family of $\kappa$ independently generated Cohen reals into a $Q$-set, whose square is proved to fail to be a $Q$-set in the extension. However, we will give a slightly different description of this model. The techniques we use for the proof are heavily inspired by the ones that appear in \cite{BrendleQ}. The proof that the resulting set is a $Q$-set is virtually the same as Brendle's, but we prove it here as well for the sake of completeness, as our notation is somewhat different than his.


We start by defining a large poset. The poset that we are really going to use will be a subset of the following:

\begin{definition}$\mathbb Q$ is the poset of all triples $(f, p, q)$ such that:

\begin{enumerate}[label=Q\arabic*)]
    \item $f\in \Fn(\kappa^2, \omega)$,
    \item $p\in \Fn(\kappa, 2^{<\omega})$,
    \item $q \in \Fn(\kappa^+, [2^{<\omega}\times\{0\}\times \omega \cup \kappa\times\{1\}\times \omega]^{<\omega})$ is such that:

\begin{enumerate}
    \item  $\forall \gamma \in \kappa^+\, \forall \xi \in \kappa\, \forall n \in \omega\, (\xi, 1, n) \in q(\gamma)\rightarrow \xi\in \dom p$,
    \item  $\forall \gamma \in \kappa^+\, \forall \xi \in \kappa\, \forall n \in \omega\, \forall s \in 2^{<\omega} (s, 0, n) \in q(\gamma) \andd (\xi, 1, n) \in q(\gamma)\rightarrow  p(\xi)\perp s$.
\end{enumerate}
\end{enumerate}

We define $(f, p, q)\leq (f', p', q')$ iff:
\begin{enumerate}[label=O\arabic*)]
    \item $f'\subseteq f$,
    \item $\dom p'\subseteq \dom p$ and for all $\xi \in \dom p'$, $p'(\xi)\subseteq p(\xi)$, and
    \item $\dom q'\subseteq \dom q$ and for all $\gamma\in \dom q'$, $q'(\gamma)\subseteq q(\gamma)$.
\end{enumerate}
\end{definition}

The coordinates $f$ and $p$ are encoding $\omega_1$ Cohen reals. Intuitively, if $G$ is $\mathbb Q$-generic:

\begin{itemize}
    \item For each $\xi<\kappa$, $x_\xi=\bigcup\{p(\xi): (f, p, q) \in t\in G\} \in 2^\omega$. Let $X=\{x_\xi: \xi<\kappa\}$.
    \item For each $n \in \omega$, $P_n=\{(x_{\rho_0} x_{\rho_1}): \exists \, (f, p, q) \in r\in G\, \rho_0, \rho_1 \in \dom f \andd f(\rho_0, \rho_1)=n\}$
    \item For each $n \in \omega$ and $\gamma<\kappa^+$, $U_n^\gamma=\bigcup\{[s]:\exists (f, p, q) \in G\, (s, 0, n) \in r(\gamma)\}$.
\end{itemize}

Condition $(a)$ is there to allow us to write condition $(b)$ and to ease the notation at some points of the proof. Condition $(b)$ is for internal consistency: this time, $(\xi, 1, n) \in r(\gamma)$ may be interpreted as ``$(f, p, q)$ promises that $x_\xi \notin U^\gamma_n$".

Here we note that $(P_n: n \in \omega)$ will be a partition of $X^2$ without point-finite open expansion. This is equivalent to saying that $X^2$ is not a $\Delta$-set since, for all $m \in \omega$, $\bigcup_{n \geq m}P_n$ compose the vanishing sequence without open expansion. $X$ will be our $Q$-set: we must guarantee that every subset of $X$ will be of the form $\bigcap_{n \in \omega} U^\gamma_n$ for some $\gamma<\kappa^+$, which does not work right now. Thus, we need to refine our order $\mathbb Q$ to a smaller order $\mathbb P$, which will be achieved below in Definition \ref{def:P}.

For the remaining of this section, we fix a bookkeeping function on $\kappa^+$, that is, a function $g$ of $\kappa^+$ onto $\kappa^+\times \kappa^+$ such that for every $\gamma<\kappa^+$, if $g(\gamma)=(\zeta, \xi)$ then $\zeta\leq \xi$. We denote the coordinates of $g$ as $g_0, g_1$. Thus, for every $\gamma<\kappa^+$, $g_0(\gamma)\leq \gamma$.

\begin{definition}\label{def:P}We recursively define subposets $\mathbb P_\gamma$ of $\mathbb Q$ for $\gamma\leq\kappa^+$ along with families $(\dot A^\gamma_\eta: \eta<\kappa^+)$ for $\gamma<\kappa^+$ such that, for all $\gamma\leq \kappa^+$:

\begin{enumerate}[label=\roman*)]
    \item $\mathbb P_0=\{(f, p, q) \in \mathbb Q: q=\emptyset\}$.
    \item If $\zeta<\gamma$, then $\mathbb P_\zeta\subseteq \mathbb P_\gamma$ and for all $(f, p, q) \in \mathbb P_\gamma$, $(f, p, q|\zeta) \in \mathbb P_\zeta$.
    \item $\mathbb P_\gamma=\bigcup_{\zeta<\gamma} \mathbb P_\zeta$ whenever $\gamma$ is limit.
    \item $\PP_\gamma$ has the Knaster property.
    \item $(\dot A_\eta^\gamma: \eta<\kappa^+)$ is the family of all $\mathbb P_\gamma$-nice names for a subset of $\check \kappa$.
    \item If $\gamma<\kappa^+$, $\mathbb P_{\gamma+1}$ is:\begin{equation*}\mathbb P_\gamma \cup \left\{(f, p, q|\gamma) \in \mathbb Q: (f, p, q)\in \mathbb P_\gamma,\,\gamma \in \dom q \andd  \forall (\xi, 1, n) \in  q(\gamma)\, (f, p, q)\Vdash_{\gamma} \check \xi \notin \dot A^{g_0(\gamma)}_{g_1(\gamma)}\right\}.\end{equation*}
\end{enumerate}

We define $\mathbb P=\mathbb P_{\kappa^+}$ and $\dot A_\gamma=\dot A^{g_0(\gamma)}_{g_1(\gamma)}$ for each $\gamma<\kappa^+$. We also write $\Vdash_{\gamma}$ instead of $\Vdash_{\gamma}$, and $\Vdash$ instead of $\Vdash_{\kappa^+}$.
\end{definition}

We show that such a construction is possible.

\begin{proof}[Construction]

We proceed by recursion in $\gamma<\kappa^+$.

In all cases, if $\gamma<\kappa^+$, after constructing $\mathbb P_\gamma$ we define $(\dot A^\gamma_\eta: \eta<\kappa^+)$ as in vi). To be able to do that, we must show that only $\kappa^+$ such nice names exist. By iii), it follows that:

$$|\mathbb P_\gamma|\leq |\{(f, p, q) \in \mathbb Q: \dom q\subseteq \gamma\}|$$$$\leq |\Fn(\kappa^2, \omega)|.|\Fn(\kappa, 2^{<{\omega}})|.|\Fn(\gamma, [2^{<\omega}\times \{0\}\times \omega \cup \kappa\times\{1\}\times \omega]^{<\omega})|=\kappa.\kappa.\kappa=\kappa.$$

Thus, for all $\gamma<\kappa^+$, $|\mathbb P_\gamma|=\kappa$. Now, by the Knaster Property, $\mathbb P_\gamma$ has the countable chain condition, so, by GCH, the number of $\mathbb P_\gamma$-nice names for subsets of $\check \kappa$ is less or equal to:

$$|([\mathbb P_\gamma]^\omega)^\kappa|\leq \kappa^\kappa=\kappa^+.$$

Now we proceed to show how to construct the sets $\mathbb P_\gamma$ in each step having defined the previous ones.

\textbf{Step 0:} We define $\mathbb P_0$ as in i). Condition v) holds by a standard $\Delta$-system argument. In fact, $\mathbb P_0$ is equivalent to adding $\kappa$ Cohen reals. We prove v) anyway for the sake of completeness:

Let $((f_\mu, p_\mu, \emptyset): \mu<\omega_1)$ be a family of elements of $\mathbb P_0$. By the $\Delta$-system Lemma applied twice, there exists $I\subseteq \omega_1$ with $|I|=\omega_1$ and a finite sets $R$, $S$ such that for all distinct $\mu, \mu' \in I$, $\dom f_\mu\cap \dom f_{\mu'}=R$ and $\dom p_\mu\cap \dom p_{\mu'}=S$. Since $\omega^R$ and $(2^{<\omega})^S$ are both countable, there exists $J\subseteq I$ with $|J|=\omega_1$ such that for all $\mu, \mu' \in J$, $f_\mu|R=f_{\mu'}|R$ and $p_\mu|S=p_{\mu'}|S$. Then, clearly, if $\mu, \mu' \in J$, $(f_\mu\cup f_{\mu'}, p_\mu\cup p_{\mu'}, \emptyset)\in \mathbb P_0$ extends both $(f_\mu, p_\mu, \emptyset)$ and $(f_{\mu'}, p_{\mu'}, \emptyset)$.

\textbf{Limit step $\gamma$}: Define $\mathbb P_\gamma$ as in iv). ii) is trivial, and by the inductive hypothesis, iii) is clear. We must prove v).

Let $((f_\mu, p_\mu, q_\mu): \mu \in \omega_1)$ be a family in $\mathbb P_\gamma$. By the $\Delta$-system Lemma, there exists $I \in [\omega_1]^{\omega_1}$ and a finite $R\subseteq \gamma$ such that for every distinct $\mu, \mu' \in  I$, $\dom q_\mu \cap \dom q_{\mu'}=R$. Let $\eta<\gamma$ be such that $R\subseteq \eta$. Since $\mathbb P_\eta$ has the Knaster property, there exists $J\in [I]^{\omega_1}$ such that for every $\mu, \nu \in I$, $(f_\mu, p_\mu, q_\mu|\eta)$ and $(f_\nu, p_\nu, q_\nu|\eta)$ are compatible. We claim that for such $\nu, \mu$, $(f_\mu, p_\mu, q_\mu)$ and $(f_\nu, p_\nu, q_\nu)$ are compatible as well.

Let $(f, p, q)\leq (f_\mu, p_\mu, q_\mu|\eta), (f_\nu, p_\nu, q_\nu|\eta)$ be a common extension in $\mathbb P_\eta$. Let $T=\dom q_\mu \cup \dom q_{\mu}\setminus \eta$. If $T=\emptyset$, we are done.
If not, let $T=\{\zeta_0, \dots, \zeta_k\}$ for some $k\in \omega$, written in increasing order. For each $i\leq k$, notice that $\xi_i \notin R=\dom q_\mu\cap \dom q_\nu$ and let $\bar q=q\cup\{(\zeta_j, q_\mu(j)): j\leq k, \zeta_j \in \dom q_\mu\}\cup\{(\zeta_j, q_\nu(j)): j\leq k, \zeta_j \in \dom q_\nu\}$.
Then, inductively, one easily verifies that $(f, p, \bar q|{\zeta_i+1})\leq (f_\mu, p_\mu, q_\mu|{\zeta_i+1}), (f_\nu, p_\nu, q_\nu|{\zeta_i+1})$. By putting $i=k$, we are done.

\textbf{Successor step $\gamma+1$:} Define $\mathbb P_{\gamma+1}$ as in vi). By inductive hypothesis, ii) is clear. We have to prove iv).

Let $((f_\mu, p_\mu, q_\mu): \mu \in \omega_1)$ be a family in $\mathbb P_{\gamma+1}$. If uncountably many elements of the family are in $\mathbb P_\gamma$, we are done. Thus, we may assume that no element of the family is in $\mathbb P_\gamma$. 

By the inductive hypothesis, there exists an uncountable $I\subseteq \omega_1$ such that all the elements of $((f_\mu, p_\mu, q_\mu|\gamma): \mu \in I)$ are compatible in $\mathbb P_\gamma$. Since $[2^{<\omega}\times \{0\}\times \omega]^{<\omega}$ is countable, there exists an uncountable $J\subseteq I$  and a set $S$ such that for all $\mu\in J$, $q_\mu\cap (2^{<\omega}\times \{0\}\times \omega)=S$. We claim that for all $\mu, \nu \in J$, $(f_\mu, p_\mu, q_\mu)$  and $(f_\nu, p_\nu, q_\nu)$ are compatible.

Let $(f, p, q)$ extend both $(f_\mu, p_\mu, q_\mu|\gamma)$ and $(f_\nu, p_\nu, q_\nu|\gamma)$ in $\mathbb P_\gamma$. Let $\bar q=q\cup\{(\gamma, q_\nu(\gamma)\cup q_\mu(\gamma))\}$. Then clearly, if $(f, p, \bar q) \in \mathbb P_\gamma$, it is a common extension of both $(f_\mu, p_\mu, q_\mu)$ and $(f_\nu, p_\nu, q_\nu)$.

$(f, p, \bar q) \in \mathbb Q:$ we have to show that if $(s, 0, n), (\xi, 1, n) \in \bar q(\gamma)$, then $\xi \in \dom p$ and $p(\gamma)\perp s$. Without loss of generality $(\xi,1, n)\in q_\mu(\gamma)$, so as $(s, 0, n)\in S\subseteq  q_\mu(\gamma)$ as well, we have $\xi \in \dom p_\mu$ and $s\perp p_\mu(\xi)$. Thus, $\xi \in \dom p$ and, as $p(\xi)\supseteq p_\mu(\xi)$.

$(f, p, \bar q) \in \mathbb P_{\gamma+1}:$ We have to show that if $(\xi, 1, n) \in \bar q(\gamma)$, then $(f, p, q)\Vdash_{\gamma} \check \xi \notin \dot A^{g_0(\gamma)}_{g_1(\gamma)}$. Without loss of generality, we have $(\xi, 1, n) \in q_\mu(\gamma)$, so $(f_\mu, p_\mu, q_\mu|\gamma)\Vdash_{\gamma} \check \xi \notin \dot A^{g_0(\gamma)}_{g_1(\gamma)}$. As $(f, p, q)\leq (f_\mu, p_\mu, q_{\mu}|\gamma)$, we are done.

\end{proof}

\begin{lemma}\label{lemma:iteration}
    For every  $\zeta<\gamma\leq \kappa^+$, then $\mathbb P_\zeta\subseteq^c \mathbb P_\gamma$ and for all $(f, p, q) \in \mathbb P_\gamma$ and for all $\zeta<\gamma$, $(f, p, q|\zeta)\in \mathbb P_\zeta$ is a reduction of $(f, p, q)$. More specifically, for every $(f', p', q') \in \mathbb P_\zeta$ below $(f, p, q|\zeta)$, the triple $(f', p', q'\cup q|{[\zeta, \gamma)]})$ is a common extension of $(f', p', q')$ and $(f, p, q)$ in $\mathbb P_\gamma$.
\end{lemma}

\begin{proof}
We prove this by induction in $\gamma$. If $\gamma=0$, there is nothing to prove.

First, assume $\gamma$ is limit.

Fix $\zeta<\gamma$. Clearly, $\mathbb P_\zeta\subseteq \mathbb P_\gamma$. We show that two elements $(f, p, q), (f', p', q') \in \mathbb P_\zeta$ that are compatible in $\mathbb P_\gamma$, are compatible in $\mathbb P_\zeta$: such a common extension $t''=(f'', p'', q'')$ 
would be in some $\mathbb P_{\eta}$ for some $\eta<\gamma$.
If $\eta\leq \zeta$ we are done as $\PP_\eta\subseteq \PP_\zeta$, and if $\eta>\zeta$, we are done as well, since $(f'', p'', q''|\gamma)$ clearly extends $(f, p, q)$ and $(f', p', q')$.

Now let $(f, p, q) \in \mathbb P_\gamma$. We must show that $(f, p, q|\zeta)\in \mathbb P_\gamma$ is a reduction of $(f, p, q)$. Assume $(f', p', q')\in \mathbb P_\zeta$ extends $(f, p, q|\zeta)$. There exists $\eta$ such that $\zeta<\eta<\gamma$ and $(f, p, q)\in \mathbb P_\eta$. As $(f, p, q|\zeta)$ is a reduction of $(f, p, q)\in \mathbb P_\zeta$ to $\mathbb P_\zeta$, we have $(f', p', q'\cup q|{[\zeta, \eta)})=(f', p', q'\cup q|{[\zeta, \gamma)})$ extends both $(f, p, q)$ and $(f', p', q')$ in $\mathbb P_\eta$, and therefore in $\mathbb P_\gamma$.

Now we prove the step $\gamma+1$. Let $\zeta\leq \gamma$ be given. If $(f, p, q)$, $(f', p', q') \in \mathbb P_\zeta$ are compatible in $\mathbb P_{\gamma+1}$, by letting $(f'', p'', q'')\in \mathbb P_{\gamma+1}$ be such a common extension, it is clear that $(f'', p'', q''|\zeta)\in \mathbb P_\zeta$ is still a common extension.

Now let $(f, p, q) \in \mathbb P_{\gamma+1}$. We must show that $(f, p, q|\zeta)\in \mathbb P_\zeta$ is a reduction of $(f, p, q)$ to $\mathbb P_\zeta$. We know it is a reduction of $(f, p, q|\gamma)$. Assume $(f', p', q')\in \mathbb P_\zeta$ extends $(f, p, q|\zeta)$. Then $t=(f', p', q'\cup q|{[\zeta, \gamma)})$ extends both $(f', p', q')$ and $(f, p, q|\gamma)$. If $q=q|\gamma$ we are done. If not, as $t\leq (f, p, q|\gamma)$, it follows that $(f', p', q'\cup q|{[\zeta, \gamma]})=(f', p', (q'\cup q|{[\zeta, \gamma)})\cup\{(\gamma, q(\gamma))\})\in \mathbb P_\gamma$ extends both $(f', p', q')$ and $(f, p, q)$.

\end{proof}

\begin{lemma}\label{lemma:augmentation}For every $(f, p, q) \in \mathbb P$:

\begin{enumerate}[label=\alph*)]
    \item If $\gamma \in \kappa^+\setminus \dom q$, let $\bar q=q\cup\{(\gamma, \emptyset)\}$. Then $(f, p, \bar q)\in \mathbb P$ and $(f, p, \bar q)\leq (f, p, q)$.

    \item If $\xi \in \kappa \setminus \dom p$ let $\bar p=p\cup\{(\xi, \emptyset)\}$. Then $(f, \bar p,  q)\in \mathbb P$ and $(f, \bar p,  q)\leq (f, p, q)$.

    \item If $\xi \in \dom p$, $t \in 2^{<\omega}$ and $p(\xi)\subseteq t$, let $\bar p=p|({\dom p\setminus \{\xi\}}\cup\{(\xi, t)\})$. Then $(f, \bar p, q)\in \mathbb P$ and $(f, \bar p, q)\leq (f, p, q)$.
    
    \item If $(\rho_0, \rho_1)\in \kappa^{2}\setminus  \dom f$ and $n \in \omega$, then $(\bar f, p, q) \in \mathbb P_\gamma$ extends $(f, p, q)$, where $\bar f=f\cup\{((\rho_0, \rho_{1}), n)\}$.

\end{enumerate}
\end{lemma}
    \begin{proof}
        Clearly, all the elements defined in this lemma are in $\mathbb Q$, and, and all inequalities are valid there. Thus, we just need to check that all these elements are in $\mathbb P$.
        
        a) Observe that $(f, p, \bar q)=(f, p,  q)\in \mathbb P_\gamma$ and that $\bar q(\gamma)$ is empty, so $(f, p, \bar q|\gamma+1)\in \mathbb P_{\gamma+1}$.
        As $(f, p, \bar q|\gamma+1)\leq (f, p, q|(\gamma+1))$, it follows from Lemma \ref{lemma:iteration} that $(f, p, \bar q|(\gamma+1)\cup q|{(\gamma+1, \kappa^+)})=(f, p, \bar q)\in \mathbb P$ extends $(f, p, q)$.
        
        b) Clearly $(f, \bar p, q|0)\in \mathbb P_0$ is below $(f, p, q|0)$, so it follows from Lemma \ref{lemma:iteration} that $(f, \bar p, q)\in \mathbb P$ and is below $(f, p, q)$

        c) and d) are similar to b).

    \end{proof}
\begin{lemma}
          For each $\gamma <\kappa^{+}$, $\xi\in \kappa$ and $n \in \omega$ the set $D^{\gamma}_{n, \xi} = \{(f,p,q) \in \PP : \gamma \in \dom q,  \xi \in \dom p \andd  (p(\xi), 0, n) \in q(\gamma) \vee (\xi, 1, n) \in q(\gamma))\}$ is dense.
\end{lemma}
         \begin{proof}
          Let $(f,p,q) \in \PP$. By the previous lemma, without loss of generality we may assume that $\gamma \in \dom q$ and that $\xi \in\dom p$. Assume $(\xi, 1, n) \not\in q(\gamma)$.
          
          Let $F=\{\mu\in \kappa^m: (\mu, 1, n)\in q(\gamma)\}$. By the definition of $\mathbb Q$, $F\subseteq \dom p$. Let $k=\max\{|p(\mu)|: \mu\in \dom p\}$.
          
          By item c) of the previous lemma applied finitely many times, it follows that $(f, \bar p, q)\in \mathbb P$ is below $(f, p, q)$, where $\bar p: \dom p\rightarrow 2^{k+1}$ is such that:

          \begin{equation}
              \bar p(\mu)(i)=\begin{cases}
                  p(\mu)(i), & \text{if } i<|p(\mu)|,\\
                0 & \text{if } |p(\mu)|\leq i<k,\\
                0, & \text{if } i=k \andd \mu\neq \xi,\\
                1, & \text{if } i=k \andd \mu= \xi.\\
               
              \end{cases}
          \end{equation}
             The first line guarantees that $\bar p$ is an extension of $p$. The second line defines an irrelevant value $0$ that could be defined as anything. The last lines serve as a marker for $\xi$.

        If $\mu\in F$, then $\mu\neq\xi$. Thus, $\bar p(\mu)(k+1)\neq \bar p(\xi)(k+1)$. Observe that $\forall \mu \in F\,  \bar p(\mu)\perp  \bar p(\xi)$.

          Now define $\bar q$ with the same domain as $q$, equal to $q$ outside of $\gamma$, and such that $\bar q(\gamma)=q(\gamma)\cup\{(\bar p(\xi), 0, n)\}$. Note that $(f, \bar p, \bar q)\in \mathbb Q$ is below $(f, p, q)$ and that $(f, \bar p, \bar q|\gamma)=(f, \bar p, q|\gamma)\in \mathbb P_\gamma$ is below $(f, p, q|\gamma)$ by Lemma \ref{lemma:augmentation} (c). $(f, \bar p, \bar q|(\gamma+1)) \in \mathbb P_{\gamma+1}$ is in $\mathbb Q$, and it easily follows it is also in $\mathbb P_{\gamma+1}$ and that it extends $(f, p, q|(\gamma+1))$, so $(f, \bar p, \bar q)\in \mathbb P$ extends $(f, p, q)$ by Lemma \ref{lemma:iteration}. This condition is clearly in $D^\gamma_{\xi, n}$, as intended.
    \end{proof}

Recall that if $\mathbb S$ is a forcing poset and $\mathbb S_0\subseteq \mathbb S$ is a subposet containing the top element, we say that $\mathbb S_0$ is completely contained in $\mathbb S$ ($\mathbb S_0\subseteq^c\mathbb S$) if the inclusion is a complete embedding. Also, recall that in this notation, if $s\in \mathbb S$, a reduction of $s$ in $\mathbb S_0$ is an element $s'\in \mathbb S_0$ such that for every $r\in \mathbb S_0$ such that $r\leq s'$, $r$ is compatible with $s$.

\begin{lemma}
    For each $\gamma<\kappa^+$ and $\xi \in \kappa$, the set $E_{\xi}^\gamma = \{(f,p,q) \in \PP : (\gamma \in \dom(q) \andd  \exists n \in \omega\, (\xi, 1, n) \in q(\alpha)) \,\vee\, (f,p,q) \Vdash_{\gamma} \check \xi \in \dot A_\gamma\}$ is dense.
\end{lemma}
\begin{proof}
    
 Let $(f,p,q) \in \PP$. Without loss of generality assume that $\xi \in \dom p$, $\gamma \in \dom q$ and that $(f,p,q) \not\Vdash_\gamma \check \xi \in \dot A_\gamma$. Then there must exist $(f',p',q') \in \PP_\gamma$ below $(f,p,q)$ such that $(f',p',q') \Vdash_\gamma \check \xi \not\in \dot A_\gamma$. We may fix $n \in \omega$ such that no element of $q'(\gamma)$ has $n$ as its third entry. Let $q''$ be defined on $\dom q\cup \dom q'$ by:
 
 \begin{equation}
     q''(\zeta)=\begin{cases}
         q'(\zeta),&\text{if } \zeta<\gamma,\\
         q(\zeta),&\text{if }\zeta>\gamma,\\
         q(\zeta)\cup\{(\xi, 1, n)\}, &\text{if } \zeta=\gamma.
     \end{cases}
 \end{equation}
 
 Now $(f',p',q''|(\gamma+1))$ satisfies the conditions established in (Q3) and (vi) from before. It follows from Lemma \ref{lemma:iteration} that $(f',p',q'') \in E_\xi^\gamma$ and $(f',p',q'') \leq (f,p,q)$.
 
\end{proof}

Throughout this paper, we will use the following notation: $\dot G=\{(p, \check p): p \in \mathbb P\}$ to describe the usual $\PP$-name for a generic filter. Moreover, by the Maximal principle, we fix the following notation:

\begin{definition}
    For every $\gamma\leq \kappa^+$:

    \begin{itemize}
        \item The $\mathbb P_\gamma$-name $\dot G_\gamma$ is the usual name for a $\mathbb P_\gamma$-filter: $\{(p, \check p): p \in \mathbb P_\gamma\}$
        \item $\dot G=\dot G_{\kappa^+}$.
        \item $\dot x$ is a $\mathbb P_0$-name, $\Vdash_{0} ``\dot x \text{ is a family of domain } \check \kappa"$ and $$\Vdash_0 \forall \xi \in \check \kappa \,\dot x_\xi=\bigcup\{p(\xi): \exists f, q\, (f, p, q)\in \dot G_0\}.$$
        
        \item $\dot X$ is a $\mathbb P_0$-name and $\Vdash_{0} \dot X=\{\dot x_\xi: \xi<\check \kappa\}.$

        \item $\dot P$ is a $\mathbb P_0$-name and $\Vdash_{0} ``\dot P \text{ is a sequence and } \forall n \in \omega\, \dot P_n=\{(\alpha, \beta)\in \check \kappa^2: \exists (f, p, q)\in \dot G_0\, f(\alpha, \beta)=n\}"$.
    \end{itemize}
\end{definition}

Of course, as for every $\gamma\leq \kappa^+$, $\mathbb P_0\subseteq^c\mathbb P_\gamma$, we conclude that $\dot x, \dot X$ and $\dot P$ are also $\mathbb P_\gamma$-names and have the same properties with respect to $\gamma$, so, for instance,         $\Vdash_{\gamma} \dot x \text{ is a family of domain } \check \kappa $ $\andd \forall \xi \in \check \kappa\, \dot x_\xi=\bigcup\{p(\xi): \exists f, q\, (f, p, q)\in \dot G_\gamma\}$. Also, for instance, if $G_\gamma$ is $\mathbb P_\gamma$-generic over $V$, $G_{0}=\mathbb G_\gamma\cap \mathbb P_\gamma$ is $\mathbb P_0$-generic over $V$ and $\text{val}(\dot x, G_0)=\text{val}(\dot x, G_\gamma)$.

\begin{lemma}\label{lemma:part}We have:

\begin{itemize}
    \item $\Vdash_0 \forall \xi<\kappa\, \dot x_\xi \in 2^\omega$
    \item $\Vdash_0 \forall \xi, \xi'<\kappa \, \xi\neq \xi'\rightarrow \dot x_\xi\neq \dot x_{\xi'}$
    \item $\Vdash_0 (\dot P_n: n \in \omega) \text{ is a partition of } \kappa^2$.
\end{itemize}
\end{lemma}

Again, of course all these claims also hold by changing $0$ to $\gamma$ for $\gamma\leq \kappa^+$.

\begin{proof}
    Consider the dense sets $\{(f, p, q): \xi\in \dom p \vee |p(\xi)|\geq n\}$, $\{(t, p, q): \xi, \xi'\in \dom p \vee p(\xi)\neq p(\xi')\}$ and $\{(f, p, q): (\alpha, \beta) \in \dom f\}$, for $\xi,\xi'<\kappa$ with $\xi\neq \xi'$, $n \in \omega$ and $(\alpha, \beta)\in \kappa^2$. 
\end{proof}

Lemma \ref{lemma:iteration} allows us to view $\mathbb P$ as an iteration. Moreover, it is easy to verify that $$\Vdash_\gamma\, \mathbb P(\dot A_\gamma, \dot X_{G_\gamma}) \text{ is densely embeddable in }P_{\gamma+1}/\mathbb P_\gamma=\{p \in \mathbb P_{\gamma+1}: \forall p' \in \dot G_\gamma\, p\not \perp p'\}.$$ We will not use this fact directly, but it helps to visualize the proof below.

\begin{theorem}
    $\Vdash \dot X \mbox{ is a Q-set}$
\end{theorem}

\begin{proof}
    It suffices to see that, given a $\mathbb{P}$-name $\tau$ such that $\Vdash \tau \subseteq \check \kappa$, we have $$\Vdash \{\dot x_\xi : \xi \in \tau\} \mbox{ is } G_\delta \mbox{ in } \dot X.$$
    
    Fix such a $\tau$ and let $\tau'$ be a nice name for a subset of $\kappa$ such that $\Vdash \tau' = \tau$. Since $\kappa^+$ is regular, $\kappa<\kappa^+$ and $\mathbb P$ has the countable chain condition, there exists $\zeta < \kappa^{+}$ such that $\tau'$ is a $\mathbb{P}_\zeta$-name. By Definition \ref{def:P} fix $\gamma, \eta < \kappa^{+}$ such that $\tau' = \dot A_\gamma=\dot A^{\zeta}_{\eta}$ and $f(\gamma) = (\zeta,\eta)$.
    
    Let $\dot U$ be a $\mathbb{P}$-name such that $$\Vdash ``\dot U \mbox{ is a sequence} \andd \forall n \in \omega \,\dot U_n = \bigcup\{[s] : \exists (f,p,q) \in \dot G\,  \gamma \in \dom(q) \andd (s,0,n) \in q(\gamma)\}".$$

    We will prove that $\Vdash \dot X\cap \bigcap_{n \in \omega} \dot U_n=\{\dot x_\xi: \xi \in \tau\}$. This is equivalent to showing that $\Vdash \dot X\setminus \bigcup_{n \in \omega} \dot U_n=\{\dot x_\xi: \xi \in \kappa\setminus\tau\}$.

    For every $\xi \in \kappa$ and $n \in \omega$, $D^{\gamma}_{n,\xi}$ is dense, which proves that $\Vdash \bigcup \dot X\setminus \bigcup_{n \in \omega} \dot U_n\subseteq\{\dot x_\xi: \xi \in \kappa\setminus\tau\}$. On the other hand, for each $\xi \in \kappa$, $E^\gamma_\xi$ is dense, which shows that $\Vdash \{\dot x_\xi: \xi \in \kappa\setminus\tau\}\subseteq \dot X\setminus \bigcup_{n \in \omega} \dot U_n$.
    
    \end{proof}

The definition below is heavily inspired by the one of similar name found in \cite{BrendleQ}. In order to ease the notation at several points of the proof of the main theorem, we will work with conditions such that $\dom f=(\dom p)^2$.
    \begin{definition}
        
        We say that a condition $(f, p, q)\in \mathbb P$ is squared if $\dom f= (\dom p)^2$.

       Fix a infinite countable set $Z=\{z_n: n \in \omega\}$, enumerated injectively, disjoint from $\kappa^+$. A pattern is a $6$-tuple $C=(\Gamma, \Delta,\bar D=(D_\gamma: \gamma \in \Gamma), \bar E=(E_\gamma: \gamma \in \Gamma), \bar \tau=(\tau_\xi: \xi \in \Delta), \bar l=(l_{\rho_0, \rho_1}: \rho_0, \rho_1 \in \Delta))$ such that:

       \begin{itemize}
           \item $\Gamma\subseteq \kappa^+$ is finite.
           \item $\Delta\subseteq \kappa\cup Z$ is finite, and $\Delta\cap Z$ is a initial segment of $Z$.
           \item $D_\gamma\subseteq 2^{<\omega}\times \omega$ is finite for each $\gamma \in \Gamma$.
           \item $E_\gamma\subseteq \Delta\times \omega$ is finite for each $\gamma \in \Gamma$.
           \item $\tau_\xi \in 2^{<\omega}$ for each $\xi \in \Delta$.
           \item $l_{\rho_0, \rho_1} \in \omega$ for each $\rho_0, \rho_1 \in \Delta$.

       \end{itemize}

       The set of all patterns is called $\PAT$.

        Let $X\subseteq \kappa^+$ and $Y\subseteq \kappa$ and $(f, p, q)\in \mathbb P$ a squared condition. We say that a pattern $(\Gamma, \Delta,\bar D, \bar E, \bar \tau, \bar l)$ is a $(X, Y)$-pattern iff there exists a (necessarily unique) bijection $\phi: \Delta\rightarrow \dom p$ such that:

               \begin{itemize}
           \item $\Gamma= X\cap \dom q$, $\Delta\subseteq Y\cup Z$, $\dom p \cap Y=\Delta \cap Y$.
           \item $\phi|_{\Delta \cap \dom p}=\phi|_{\Delta\cap Y}$ is the inclusion in $\dom p$, $\phi|_{\Delta\cap Z}\rightarrow \dom p\setminus Y$ is an order isomorphism.
            \item For all $\xi \in \Delta$, $\tau_\xi=p(\phi(\xi))$.
            \item For all $\rho_0, \rho_1 \in \Delta$, $l_{\rho_0, \rho_1}=f(\phi(\rho_0), \phi(\rho_1))$.
            \item For all $\gamma\in \Gamma$, $(s, n)\in D_\gamma$ iff $(s, 0, n)\in q(\gamma)$.
            \item For all $\gamma\in \Gamma$, $(\xi, n)\in E_\gamma$ iff $(\phi(\xi), 1, n)\in q(\gamma)$.
       \end{itemize}

    \end{definition}
    It is easy to verify that squared conditions are dense, and that for every pair $(X, Y)$ as above, every squared condition $(f, p, q)$ has an unique $(X, Y)$-pattern for it.

    \begin{theorem}\label{thm:mainforcing}
        $\Vdash \dot X^2 \text{ is not a } \Delta\text{-}\text{set}$.
    \end{theorem}

    \begin{proof}
        Let $\tilde P$ be a $\mathbb P$-name such that:
        $\Vdash ``\tilde P \text{ is a sequence and } \forall n \in \omega\, \tilde P_n=\{(\dot x_\alpha, \dot x_\beta): (\alpha, \beta) \in \dot P_n\}"$. By Lemma \ref{lemma:part}, $\Vdash ``\tilde P \text{ is a partition of } \dot X^2"$. We will show that $\tilde P$ is forced to not have a point finite open expansion.

        It suffices to show that for every name $\dot V$ such that $$\Vdash ``\dot V \text{ is a sequence of open subsets of } (2^\omega)^2 \andd \forall n \in \omega\, \tilde P_n \subseteq \dot V_n",$$ there exists $\alpha, \beta \in \kappa$ such that $\Vdash (\dot x_{\check \alpha}, \dot x_{\check \beta})\in \bigcap_{n \in \omega} \dot V_n$.

        For each $n \in \omega$, let $\dot S_n$ be a name such that $\Vdash \dot S_n=\{(s_0, s_1) \in (2^{<\omega})^2: [s_0]\times [s_1]\subseteq \dot V_n\}$. For each $n \in \omega$ and $s_0, s_1\in 2^{<\omega}$, let $\mathcal A^n_{s_0, s_1}$ be a maximal antichain in $\mathbb P$ such that for every $t^* \in A^n_{s_0, s_1}$, either $t^*\Vdash (\check s_0, \check s_1) \in \dot S_n$ or $t^*\Vdash (\check s_0, \check s_1) \notin \dot S_n$.

    Let $\chi>\kappa^+$ be a regular uncountable cardinal large enough so $\mathbb P, Z, \PAT, (\dot A^n_{s_0, s_1}: n \in \omega, s_0, s_1\in 2^{<\omega})$ are all elements of $H(\chi)$. Let $M$ be a countable elementary submodel of $H(\chi)$ so that all these same elements and $\kappa$ are members of $M$.

    Let $o(M)$ be the first ordinal not in $M$. Let $\alpha, \beta$ be ordinals such that $o(M)\leq \alpha<\beta<\omega_1$. We claim that:

    $$\Vdash (\check \alpha, \check \beta) \in \bigcap_{n \in \omega}\dot U_n.$$

    We will prove that for every  $(f, p, q)\in \mathbb P$ and every $N \in \omega$, there exists $(\bar f, \bar p, \bar q)\leq (f, p, q)$ and $s_0, s_1 \in 2^{<\omega}$ such that $\alpha, \beta \in \dom \bar p$ and  $s_0\subseteq \bar p(\alpha)$, $s_1\subseteq \bar p(\beta)$ and $(\bar f, \bar p, \bar q)\Vdash (\check s_0, \check s_1) \in \dot S_N$. This ends the proof by a standard density argument.

    Fix $t=(f, p, q)$ and $N \in \omega$. Without loss of generality, $t$ is a squared condition with $\alpha, \beta \in \dom p$.
    
    Let $X=\dom q \cap M$, $Y=\dom p\cap M$. Let $C=(\Gamma, \Delta,\bar D, \bar E, \bar \tau, \bar l)$ be the unique $(X, Y)$-pattern for $(f, p, q)$. We have the inclusion $(X, Y)\in M$ by finiteness, and $C$ is in $M$ as well. By elementarity, there exists a squared condition $t'=(f', p', q')$ in $M$ such that $(X, Y)$ is a pattern for $t'$. Let $\phi_t, \phi_{t'}$ be the functions that verify that $C$ is a $(X, Y)$-pattern for $t$, $t'$, respectively. By elementarity, $\phi_{t'}\in M$.

    As $\alpha\notin M$, there exists $i$ such that $z_i \in \Delta\cap Z$ and $\phi_t(z_i)=\alpha$. Let $\alpha'=\phi_{t'}(z_i)$.

    $\alpha' \notin \dom p$, for if it was, it would be in $Y=M\cap \dom p$, so $\phi_{t'}(z_i)=\alpha'$, by definition, and $\phi_{t'}(\alpha')=\alpha'$, as $\alpha'\in \dom p'\cap Y$ and $\phi_{t'}$ is the identity in there, contradicting the injectivity of $\phi'_t$. On the other hand, $\beta \notin \dom p'$, as $\beta \notin M$ and $\dom p'\subseteq M$. Thus, $(\alpha', \beta)\notin \dom f\cup \dom f'$ since both domains are squares.

    $p'$ and $p$ agree on $\dom p \cap \dom p'$: if $\xi \in \dom p'\cap \dom p$, we have $\xi \in M\cap \dom p=Y$, so $p(\xi)=p(\phi_t(\xi))=\tau_\xi=p'(\phi_{t'}(\xi))=p'(\xi)$.

    $f'$ and $f$ agree on $\dom f\cap \dom f'$: if $(\rho_0, \rho_1) \in \dom f\cap \dom f'=(\dom p\cap \dom p')^2$, then $\rho_0, \rho_1 \in \dom p'\cap \dom p\subseteq M\cap \dom p=Y$, so $f(\rho_0, \rho_1)=f(\phi_t(\rho_0), \phi_t(\rho_1))=l_{\rho_0, \rho_1}=f'(\phi_{t'}(\rho_0), \phi_{t'}(\rho_1))=f'(\rho_0, \rho_1)$.

    If $\gamma\in \dom q\cap \dom q'$, then $q(\gamma)\cap [2^{<\omega}\times \{0\}\times \omega]^{<\omega}=q'(\gamma)\cap [2^{<\omega}\times \{0\}\times \omega]^{<\omega}$: if $\gamma \in \dom q\cap \dom q'$, then $\gamma \in \dom q\cap M=Y$, so for every $s \in 2^{<\omega}$, $\xi \in \kappa$ and $n \in \omega$, $(s, 0, n) \in q(\gamma)\iff (s, n)\in D_\gamma \iff (s, 0, n)\in q(\gamma')$.

    Now we show that $t$ and $t'$ are compatible by constructing $t''=(f'', p'', q'')$ such that:

    \begin{itemize}
        \item $\dom f''=\dom f\cup \dom f'\cup \{(\alpha', \beta)\}$,
        \item $f''(\rho_0, \rho_1)=\begin{cases}N&\text{if } (\rho_0, \rho_1)=(\alpha', \beta),\\f(\rho_0, \rho_1)&\text{if } (\rho_0, \rho_1) \in \dom f,\\f'(\rho_0, \rho_1)&\text{if }(\rho_0, \rho_1) \in \dom f'.\end{cases}$
        \item $\dom p''=\dom p\cup \dom p'$,
        \item $p''(\xi)=\begin{cases}f(\xi)&\text{if }\xi \in \dom p,\\ p'(\xi) &\text{if } \xi \in \dom p'.\end{cases}$
        \item $\dom q=\dom q\cup \dom q'$
        \item $q''(\gamma)=\begin{cases}q(\gamma)& \text{if }\gamma \in \dom q\setminus \dom q, \\ q'(\gamma)& \text{if } \gamma \in \dom q'\setminus \dom q,\\q'(\gamma)\cup q(\gamma)&\text{if }\gamma \in \dom q\cap \dom q'.\end{cases}$
    \end{itemize}

    With a straightforward induction, one shows that $(f'', p'', q''|\gamma)\in \mathbb P_\gamma$ and extends both $(f, p, q|\gamma)$ and $(f', p', q'|\gamma)$ for every $\gamma\leq \kappa^+$.

    Now, $t''\Vdash (\alpha', \beta)\in \dot P_{\check N}$, so $t''\Vdash (\dot x_{\alpha'}, \dot x_\beta)\in \dot V_{\check N}$.

    Thus, there exists $\tilde t=(\tilde f, \tilde p, \tilde q)\leq t''$ and $s_0, s_1 \in 2^{<\omega}$ such that $\tilde t \Vdash (\dot x_{\alpha'}, \dot x_\beta)\in [\check s_0]\times [\check s_1]\subseteq \dot V_N$,
    thus, extending it if necessary, we can consider that $s_0\subseteq \tilde p(\alpha')$, $s_1\subseteq \tilde p(\beta)$ and $\tilde t\Vdash (\check s_0, \check s_1)\in \dot P_N$. There exists $t^*=(f^*, p^*, q^*)\in \mathcal A^N_{s_0, s_1}$ such that $t^*\not \perp \tilde t$. As $\tilde t\Vdash (\check s_0, \check s_1)\in \dot P_N$, we also have $t^*\Vdash (\check s_0, \check s_1)\in \dot P_N$. By extending $\tilde t$ even further, we may assume that $\tilde t\leq t^*$.

    Notice that $\tilde t$ has all the requirements needed for $\tilde t$, besides possibly $s_0 \subseteq \tilde p(\alpha)$ - but we have the inclusion $s_0\subseteq \tilde p(\alpha')$.
    
    Next, we use $\tilde t$ to construct a condition $\bar t=(\bar f, \bar p, \bar q)$ such that $\bar t\leq t^*, t$, $\bar p(\beta)=\tilde p(\beta)$ and  $\bar p(\alpha)=\tilde p(\alpha')$. This will finish the proof, as we will have $\bar t\leq t$, $\bar t\Vdash (\check s_0, \check s_1)\in \dot S_N$ (as $\bar t\leq \bar t^*$), $s_0\subseteq \tilde p(\alpha')=\bar p(\alpha)$ and $s_1\subseteq \tilde p(\beta)=\bar p(\beta)$. 
    
    We define $\bar t$ by the following sentences:

    \begin{itemize}
        \item $\dom \bar f=\dom f \cup \dom f^*$,
        \item $\bar f(\rho_0, \rho_1)=\begin{cases}
            f^*(\rho_0, \rho_1), & \text{if } (\rho_0, \rho_1) \in \dom f^*,\\
            f (\rho_0, \rho_1), & \text{if } (\rho_0, \rho_1) \in \dom f\setminus \dom f^*.    
        \end{cases}$
        
        \item $\dom \bar p=\dom p \cup \dom p^*$.
        \item $\bar p(\xi)=\begin{cases}\tilde p(\xi), & \text{if } \xi\neq \alpha,\\
         \tilde p(\alpha'), & \text{if } \xi=\alpha.
            
        \end{cases}$
        \item $\dom \bar q=\dom q \cup \dom q^*$.
        \item $\bar q(\gamma)=\begin{cases}            q(\gamma), & \text{if }\gamma \in \dom q \setminus \dom q^*,\\            q^*(\gamma), & \text{if }\gamma \in \dom q^*\setminus \dom q,\\            q(\gamma)\cup q^*(\gamma) & \text{if }\gamma \in \dom q^*\cap \dom q.        \end{cases}$
        \end{itemize}

        Such a triple $\bar t$ is well defined, but we must verify that $\bar t\in \mathbb P$ and that $\bar t\leq t, t^*$.
        Notice that as $\alpha\neq \beta$,  $\bar p(\beta)=\tilde p(\beta)$ and that $\bar p(\alpha)=\tilde p(\alpha')$, as intended.

        Thus, recursively, we show that for every $\gamma\leq \kappa^+$, $(\bar f, \bar p, \bar q|\gamma)\in \mathbb P_\gamma$ and extends $(f, p, q|\gamma)$ and $(f^*, p^*, q^*|\gamma)$. The limit step is trivial, so we show how to verify this for $\gamma=0$ and for the successor step.\vspace{1em}

        \textbf{Step $\gamma=0$}: $(\bar f, \bar p, \emptyset)$ is clearly in $\mathbb P_0$. If $(\rho_0, \rho_1) \in \dom f\cap \dom f^*$, then $f^*(\rho_0, \rho_1)=\tilde f(\rho_0, \rho_1)=f(\rho_0, \rho_1)$, so $f\cup f^*\subseteq \bar f$.

        For every $\xi \in \dom p^*$, $\xi \neq \alpha$, so $p^*(\xi)\subseteq \tilde p(\xi)=\bar p(\xi)$.  Now suppose $\xi \in \dom p$. We must show that $p(\xi)\subseteq \bar p(\xi)$.
        
        \textbf{Case 1:} if $\xi\neq \alpha$, we have $\bar p(\xi)=\tilde p(\xi)\supseteq p(\xi)$.

        \textbf{Case 2:} if $\xi= \alpha$, we have $\bar p(\xi)=\bar p(\alpha)=\tilde p(\alpha')\supseteq p'(\alpha')=p'(\phi_{t'}\circ \phi_t^{-1}(\alpha))=\tau_{\phi_t^{-1}(\alpha)}=p(\alpha)$.\vspace{1em}

        \textbf{Step $\gamma+1$}: Assume that $(\bar f, \bar p, \bar q|\gamma)\in \mathbb P_\gamma$ extends $(f, p, q|\gamma)$ and $(f^*, p^*,  q^*|\gamma)$.

        We will show that for all $s\in 2^{<\omega}$, $\xi\in \kappa$ and  $n \in \omega$, if $(s, 0, n), (\xi, 1, n) \in \bar q(\gamma)$, we have $\xi \in \dom \bar p$ and $\bar p(\xi)\perp s$. This will guarantee that $(\bar f, \bar p, \bar q|_{\gamma+1})\in \mathbb Q$. Then it easily follows that it is in $\mathbb P_\gamma$ as well.
        
        Our claim is trivial if $\gamma \in \dom q\setminus \dom q^*$ or $\dom q^*\setminus \dom q$. So assume $\gamma \in \dom q\cap \dom q^*\subseteq X$. Then $\gamma \in \Gamma$. Again, this is trivial if both triples are in $q(\gamma)$, or in $q^*(\gamma)$. So we have only two cases:

        \textbf{Case A:} $(\xi, 1, n)\in q^*(\gamma)$ and $(s, 0, n)\in q(\gamma)$. In this case, $\xi \in \dom p^*$, so $\xi\neq \alpha$, therefore $\bar p(\xi)=\tilde p(\xi)\perp s$ as $(s, 0, n), (\xi, 1, n)\in \tilde q(\gamma)$.
        
        \textbf{Case B:} $(\xi, 1, n)\in q(\gamma)\setminus q^*(\gamma)$, $(s, 0, n)\in q^*(\gamma)$. We break this case into two subcases:
        
        \textbf{Subcase b1:} $\xi=\alpha$. In this case, $(\xi, 1, n)=(\alpha, 1, n)\in q(\gamma)$, so $(\phi_{t'}\circ \phi_t^{-1}(\alpha), 1, n)=(\alpha', 1, n)\in q'(\gamma)\subseteq \tilde q(\gamma)$. Also, $(s, 0, n)\in \tilde q(\gamma)\supseteq q^*(\gamma)$. Thus, $\bar p(\alpha)=\tilde p(\alpha')$ is incompatible with $s$.

        \textbf{Subcase b2:} $\xi\neq \alpha$. In this case, $(s, 0, n), (\xi, 1, n)\in \tilde q(\gamma)\supseteq q^*(\gamma)\cup q(\gamma)$, and $\bar p(\xi)=\tilde p(\xi)$. Thus, $\bar p(\xi)=\tilde p(\xi)$ is incompatible with $s$.
        
    \end{proof}

\section{On the existence of strong \texorpdfstring{$\Delta$}{Delta}-sets}\label{sec:strongdelta}

In this section, we prove that the existence of $\Delta$-sets implies the existence of strong $\Delta$-sets, that is, of $\Delta$-sets whose all finite powers are also $\Delta$-sets. We start by fixing the following notation:

\begin{definition}\label{def:algebranotation}
    Let $\mathcal A$ be a collection of sets. We define $\mathcal A_u$, $\mathcal A_i$, $\mathcal A_\delta$ $\mathcal A_\sigma$ as the collection of all finite unions, finite intersections, countable unions and finite intersections of $\mathcal A$, respectively. If all elements of $\mathcal A$ are contained in some set $Y$ which is clear from the context, $\mathcal A^c$ is the collection of all complements of members of $\mathcal A$ with respect to $Y$.
    
    If $(\mathcal A_i: i\leq r)$ is a sequence of collection of sets, $\bigotimes_{i\leq r} \mathcal{A}_i=\{\prod_{i\leq r}A_i: \forall i \leq r\, A_i \in \mathcal A_i\}$.
    
\end{definition}

\begin{definition}[\cite{szpilrajn1938characteristic}]\label{def:chi}
    Given a countable family of subsets of $\omega_1$, $\mathcal A = (A_n:n\in\omega)$, the characteristic function for the family $\mathcal{A}$ is $\chi_{\mathcal A}: \omega_1 \rightarrow 2^\omega$ given by, for every $\alpha \in \omega_1$ and $n \in \omega$ $\chi_{\mathcal A}(\alpha)(n) = \chi_{A_n}(\alpha)$.
\end{definition}

\begin{lemma}\label{lemma:chiprop}
    In the conditions of the definition above, for any $A_n \in \mathcal A$ the set $\chi_{\mathcal A}[A_n]$ is a clopen subspace of $\chi_{\mathcal A}[\omega_1]$. Moreover, the function $\chi_{\mathcal A}$ is injective if, and only if $\mathcal A$ separates points, that is, if for every $\alpha, \beta \in \omega_1$ there exists $n \in \omega$ such that $\alpha \in A_n$ and $\beta \notin A_n$.
\end{lemma}

Indeed it is immediate that $\chi_{\mathcal A}[A_n] = \{f \in 2^\omega: f(n)=1\} \cap \chi_{\mathcal A}[\omega_1]$ and therefore is a clopen.

\begin{definition}\label{def:ifunc}
    Let $r \in \omega$ and $i\leq  n$. A relation $f \subset \omega_{1}^{r+1}$ is said to be an $i$-function into $\omega_{1}^r$ if it is a function with respect to the $i$-th coordinate, that is, if for every $\alpha \in \omega_1$, $|\pi_i^{-1}[\{\alpha\}]\cap f|=1$

    If $f$ is an $i$-function into $\omega_1^r$, for every $\alpha\in \omega_1$, the unique $r$-tuple $(\beta_0, \dots, \beta_{r-1})$ such that $(\beta_j)_{j<i}^\frown(\alpha)^\frown(\beta_j)_{i\leq j<r} \in f$ is denoted by $f(\alpha)$, and such a $\beta_j$ is denoted by $f(j, \alpha)$. 
\end{definition}

The next two lemmas state useful properties of this class of objects for our purposes.

\begin{lemma}\label{lemma:ifuncdelta}
    Let $r \in \omega$, $i \leq r$ and $f$ an $i$-function into $\omega_1^{r}$. There exists a countable collection $\mathcal B\subseteq \mathcal P(\omega_1)$ such that $f \in (\bigotimes_{j\leq r} \mathcal B)_{u\delta}$.
\end{lemma}

\begin{proof}
    Fix an injective family $(x_{\beta}: \beta \in \omega_1)$ of real numbers. Then for every $\alpha$, $f(\alpha) = \vec \beta = (\beta_0, \dots, \beta_{r-1})$ iff $\forall q \in Q\, \forall j <r$, $q<x_{f(j, \alpha)}\leftrightarrow q<x_{\beta_j}$. Thus:

    First consider the case where $i = 0$ to ease our notation, the general case can be verified by ``shifting the coordinates" on any $i$-function to obtain an $0$-function. Observe that $$g = \bigcap_{j<r}\bigcap_{q \in \mathbb{Q}}\pi_0^{-1}[\{\alpha < \omega_1 : q \geq x_{f(j,\alpha)}\}]\cap \pi_{j+1}^{-1}[\{\beta < \omega_1 : q < x_{\beta}\}]$$$$\cup \pi_0^{-1}[\{\alpha < \omega_1 : q \geq x_{f(j,\alpha)}\}]\cap \pi_{j+1}^{-1}[\{\beta < \omega_1 : q < x_{\beta}\}]$$ Taking $\mathcal B$ to be the collection of the sets used to write the  intersections above we only need to verify that $f = g$ to conclude the lemma.

    Indeed let $(\alpha, \beta_0, \dots, \beta_{r-1}) \in f$. For all $j<r$ and $q \in \mathbb Q$, since either $q < f(j,\alpha) = \beta_{j-1}$ or $f(j,\alpha) = \beta_{j-1} \leq q$, it is immediate that $(\alpha, \beta_0, \dots, \beta_{r-1}) \in g$. Now, if $(\alpha, \beta_0, \dots, \beta_{r-1}) \notin f$, there are $j<r$ and $q \in \mathbb Q$ such that $f(j,\alpha) < q < \beta_{j-1}$ or $f(j,\alpha) > q >\beta_{j-1}$. Hence $(\alpha, \beta_0, \dots, \beta_{r-1}) \notin g$.
\end{proof}

\begin{lemma}\label{lemma:ifunccounting}
    Given $r<\omega$. There exists $(f^{i}_n: n \in \omega, i\leq r)$ a collection of $i$-functions into $\omega_1^{r}$ such that $\omega_1^{r+1}=\bigcup_{i\leq r}\bigcup_{n \in \omega} f^{i}_n$.
\end{lemma}

\begin{proof}
    Define, for each $i < r$, $f^{i}_n(\alpha)$ such that the set $\{f^{i}_n(\alpha) : n \in \omega\} = (\alpha+1)^{r}$. The union of the functions is a subset of $\omega_{1}^{r+1}$ by Definition \ref{def:ifunc}. Now, given $x = (x_0, \dots, x_r) \in \omega_1^{r+1}$, let $j \leq r$ be such that $x_j \geq x_i$ for all $i \leq r$. Since $\bar x = (x_0, \dots, \hat{x_j}, \dots, x_r) \in (x_j + 1)^{r}$, we must have $n \in \omega$ such that $f^{j}_n(x_j) = \bar x$, therefore $x \in f^j_n$. 
\end{proof}
The following lemma is inspired by \cite{szpilrajn1938characteristic}.

\begin{lemma}\label{lemma:identifying}
If there exists an uncountable $\Delta$-set, then there exists a countable subalgebra $\mathcal A'$ of $\mathcal P(\omega_1)$ which separates points such that for every vanishing sequence $(F_n: n \in \omega)$, there exists a vanishing sequence $(G_n: n \in \omega)$ in $\mathcal A'_\sigma$ such that $F_n\subseteq G_n$.
\end{lemma}

\begin{proof}
    Every subset of a $\Delta$-set is a $\Delta$-set, thus there exists a $\Delta$-set $X$ of size $\omega_1$. Clearly, by identifying $X$ with $\omega_1$, we may take $\mathcal A'$ to be the algebra generated by a countable basis of $X$.
\end{proof}

We are now ready to prove the main result of this section.

\begin{theorem}\label{thm:strongDelta}
    If there exists an uncountable $\Delta$-set, then there is an uncountable strong $\Delta$-set.
\end{theorem}

\begin{proof}

Since there exists an uncountable $\Delta$-set, there exists a countable set $\mathcal A'$ of subsets of $\omega_1$ as in Lemma \ref{lemma:identifying}.

For each $r<\omega$ and $i\leq r$ let $(f^{ri}_n: n \in \omega)$ be a collection of $i$-functions of $\omega_1^{r+1}$ such that $\omega_1^{r+1}=\bigcup_{i\leq r}\bigcup_{n \in \omega} f^{r, i}_n$ as seen in Lemma \ref{lemma:ifunccounting}.

For each $r<\omega$, let $g_{n}^{ri} = f_{n}^{ri} \setminus\left( \bigcup_{n'<n}\bigcup_{i'\leq r}f_{n'}^{ri'}\cup \bigcup_{i'<i}f_n^{ri'}\right)$, a partial $i$-function. By the Lemma \ref{lemma:ifuncdelta}, for each $n, r<\omega$ and $i\leq r$ there exists a countable collection $\mathcal B_{n}^{ri}$ of subsets of $\omega_1$ such that $f_n^{ri} \in (\mathcal B_{n}^{ri}\otimes\mathcal B_{n}^{ri})_{u\delta}$. Let $\mathcal A$ be the countable algebra generated by the countable collection $\bigcup_{r \in \omega}\bigcup_{i\leq r}\bigcup_{n \in \omega} \mathcal B_{n}^{ri}\cup \mathcal A'$, Then, for every $n, r \in \omega$, and $i\leq r$, $g_n^{ri}$ are intersections of a member of $(\bigotimes_{i\leq r}\mathcal A)_{u\delta}$ and a member of $(\mathcal [\bigotimes_{i\leq r} \mathcal A]^c)_{i\sigma i}$.

Consider $\vec{\mathcal A} = (A_n : n \in \omega)$ to be an enumeration of  $\mathcal A$. Let $\sigma = \chi_{\vec{\mathcal A}} : \omega_1 \rightarrow 2^{\omega}$ as in Definition \ref{def:chi} and $X = \sigma[\omega_1]$. For each $r \in \omega$, let $\sigma_r : \omega_1^{r+1} \rightarrow (2^{\omega})^{r+1}$ be given by $\sigma_r(\alpha_0, \dots, \alpha_r) = (\sigma(\alpha_0), \dots, \sigma(\alpha_r))$. Let $X=\sigma[\omega_1]$.\vspace{1em}
 
 \textbf{Claim:} For every $r,n \in \omega$ and $i\leq n$, $\sigma_r[g_n^{ri}]$ is an $F_{\sigma}$ of $X^{r+1}$.
 \begin{proof}
     Fix $n$. We prove that $\sigma_r[g_n^{ri}]$ is a $G_\delta^{ri}$ and an $F_\sigma$.

     Write $g_n^{ri} = Q_0 \cap Q_1$, where $Q_0 \in (\bigotimes_{i\leq r}\mathcal A)_{u\delta}$ and $Q_1=(\mathcal [\bigotimes_{i\leq r} \mathcal A]^c)_{i\sigma i}$.

    It suffices to show that $\sigma_r[Q_0]$ is closed and that $\sigma_r[Q_1]$ is open, since we are working with subspaces of $(2^{\omega})^r$.

     \textbf{$\sigma_r[Q_0]$ is closed:} write $Q_0=\bigcap_{l \in \omega}\bigcup_{j<l}\prod_{i'\leq r}A_{i'j l}$, where each $A_{i'j l}$ is in $\mathcal A$. Then $\sigma_r [Q_0]=\bigcap_{l \in \omega}\bigcup_{j<l}\prod_{i'\leq r}\sigma[A_{i'j l}]$. By Lemma \ref{lemma:chiprop}, each $\sigma[A_{i'j l}]$ is a clopen subset of $X$, so this is a closed set.

     \textbf{$\sigma_r[Q_1]$ is open:} write $Q_1=\bigcap_{k<p}\bigcup_{l \in \omega}\omega_1^{r+1}\setminus\bigcap_{j<l}\prod_{i'\leq r} A^k_{i' j l}$, where each $A^k_{i' j l}$ is in $\mathcal A$. Then $\sigma_r [Q_1]=\bigcap_{i<p}\bigcup_{l \in \omega}\bigcap_{j<l}X^{r+1}\setminus\prod_{i'\leq r} \sigma[A^k_{i' j l}]$. Again each $\sigma[A^k_{i' j l}]$ is a clopen subset of $X$, so this is an open set.
 \end{proof}

 Fix $r<\omega$. We show that $X^{r+1}$ is a $\Delta$-set. Let $(K_m: m \in \omega)$ be a vanishing sequence in $X^{r+1}$, and $F_m = \sigma_r^{-1}[K_m]$. For every $m, n \in \omega$, let $X^{ni}_{m} = \pi_i[g_n^{ri}\cap F_m]$. Now we have: 
 \begin{itemize}
     \item $g_n^{ri} \cap F_m = \pi_i^{-1}[X^{ni}_{m}] \cap g_n^{ri}$;

     \item For each $n \in \omega$ and $i\leq n$, $(X^{ni}_{m} : m \in \omega)$ vanishes;
     \item For each $n \in \omega$ and $i\leq n$, $\exists (G^{ni}_{m} : m \in \omega)$ vanishing expansion of $(X^{ni}_{m} : n \in \omega)$ in $\mathcal A_{\sigma}$.

    Then, for each $m \in \omega$:

    $$F_m=\bigcup_{i\leq r}\bigcup_{n \in \omega}F_m\cap g^{ri}_n=\bigcup_{i\leq r}\bigcup_{n \in \omega}\pi_i^{-1}[X^{ni}_{m}] \cap g_n^{ri}$$
    $$\subseteq\bigcup_{i\leq r}\bigcup_{n \in \omega}\pi_i^{-1}[G^{ni}_{m}] \cap g_n^{ri}\stackrel{\text{def}}=W_m$$
 \end{itemize}

 \textbf{Claim:} $(W_m: m \in \omega)$ is vanishing.

 \begin{proof}
     Assume by contradiction it is not. Fix $\vec \alpha \in \bigcap_{m \in \omega}W_m$. There exists a unique member of $(g_n^{ri}: n \in \omega,i\leq r)$ such that $\vec \alpha$ is inside of it. Then $\vec \alpha \in \bigcap_{m \in \omega}\pi_i^{-1}[G^{ni}_{m}] \cap g_n^{ri}$. Thus, $\vec\alpha(i) \in \bigcap_{m \in \omega} G_m^k=\emptyset$, a contradiction.
 \end{proof}

 Thus:

 $$K_m\subseteq \sigma_r[W_m]=\bigcup_{i\leq r}\bigcup_{n \in \omega} \pi_i^{-1}[\sigma[G^{ni}_m]]\cap  \sigma_r[g_n^{ri}].$$

 The sequence $\sigma_r[W_m]$ is vanishing.  As $(g_n^{ri}: n \in \omega, i\leq r)$ is a partition of $\omega_1^{r+1}$, its complement is:

 $$\bigcup_{i\leq r}\bigcup_{n \in \omega} \sigma_r[g_n^{ri}]\setminus \pi^{-1}_i[\sigma[G^{ni}_m]],$$

 which is clearly an $F_\sigma$ as each $\sigma[G_m^{ni}]$ is open for each $n<\omega$, $i\leq r$.
\end{proof}

Since this property is shared with $Q$-sets, we find the following question interesting:

\begin{question}
    Does the existence of a $\Delta$-set imply the existence of a $Q$-set?
\end{question}

\section{Concluding Remarks}

%

To conclude this paper we would like to put some of our remaining thoughts and leftover questions. In Section \ref{sec:forcing} we constructed an example of a $Q$-set whose square is not a $\Delta$-set. A very natural question would be if we could lift the powers in this construction.

\begin{question}\label{que:isnotpower}
    Is there a consistent example of a $Q/\Delta$-set $X$ and $n \geq 2$ such that $X^n$ is a $Q/\Delta$-set but $X^{n+1}$ is not?
\end{question}

Some natural attempts at adapting the forcing from Section \ref{sec:forcing} to increase the dimensions of the objects involved and then mimicking the proof of Theorem \ref{thm:mainforcing} seem to fail. Thus, stepping up Theorem \ref{thm:mainforcing} to higher dimensions seems to be a non-trivial problem. This segues naturally to the next question which is also tied to Section \ref{sec:strongdelta}.

\begin{question}\label{que:ispower}
    If $X$ is a $(Q/\Delta)$-set such that $X^2$ is also a $(Q/\Delta)$-set, then must $X$ be strongly $(Q/\Delta)$?
\end{question}

In this sense Theorem \ref{thm:strongDelta} guarantees that the existence of a strongly $\Delta$-set depends only on the existence of a $\Delta$-set. This may be a promising lead on whether the previous question has a positive answer. 
\section*{Acknowledgements}

We would like to thank M. Hru\v{s}ák for suggesting the model used in Section \ref{sec:forcing}, and P. Szeptycki for his comments and suggestions.

The first author is currently an unpaid researcher with no academic affiliation. He thanks the financial support of his family, in particular of his grandmother Leonor Rey and his mother Maria Inez Rey.

The second author is a voluntary (unpaid) postdoctoral researcher at the University of São Paulo. He thanks the University for the hospitality. He thanks the financial support of his parents Alberto Rodrigues, Tânia de Oliveira and his wife Bruna Nagano.


\end{document}